\documentclass{article}
\usepackage[utf8]{inputenc}

\usepackage[english]{babel}

\usepackage[utf8]{inputenc}
\usepackage[T1]{fontenc}

\usepackage{amssymb}
\usepackage{stmaryrd}
\usepackage{amsmath,amsfonts}
\usepackage[tikz]{bclogo}
\usepackage[top=4.34cm]{geometry}

\usepackage{lmodern}
\usepackage{textcomp}
\usepackage{mathtools, bm}
\usepackage{amssymb, bm}
\usepackage[unq]{unique}
\usepackage[shortlabels]{enumitem}
\usepackage{comment}
\usepackage[breaklinks]{hyperref}

\usepackage[numbers]{natbib}  

\usepackage[breaklinks]{hyperref}
\usepackage[numbers]{natbib}

\usepackage{amsthm}

\usepackage{cleveref}
\usepackage[utf8]{inputenc}
\usepackage[T1]{fontenc}
\usepackage{enumitem}
\usepackage{tipa}
\usepackage{float}
\usepackage{dirtytalk}
\usepackage{pbox}
\usepackage{xcolor}
\usepackage{setspace}
\usepackage[normalem]{ulem}
\usepackage{geometry}
\usepackage{amsmath}
\usepackage{amsthm}
\usepackage{tikz}
\usepackage{amssymb}
\usepackage{multicol}
\usepackage{microtype}
\usepackage{datetime}
\usepackage{endnotes}
\usepackage{babel}
\usepackage[mathscr]{euscript}
\usepackage{bigfoot}
\usetikzlibrary {arrows.meta}
\usepackage{bbm}

\newtheorem{theorem}{Theorem}[section]
\newtheorem{lemma}[theorem]{Lemma}

\newtheorem{conjecture}[theorem]{Conjecture}
\newtheorem{Prop}[theorem]{Proposition}
\newtheorem{Cor}[theorem]{Corollary}
\newtheorem{question}[theorem]{Question}
\newtheorem{remark}[theorem]{Remark}
\newtheorem{example}[theorem]{Example}

\theoremstyle{definition}

\newcommand{\C}{\mathbb{C}}
\newcommand{\fly}[0]{\mathcal{F}}
\newcommand{\groundset}[1]{\mathbb{Z}_{#1}}
\newcommand{\groundsubs}[2]{\groundset{#1}^{(#2)}}
\newcommand{\hquad}{\hspace{0.5em}}

\DeclareMathOperator{\id}{id}
\DeclareMathOperator{\sgn}{sgn}
\DeclareMathOperator{\Hom}{Hom}

\title{The wreath matrix}
\author{Jan Petr \and Pavel Turek}
\date{}

\begin{document}

\maketitle

\begin{abstract}
    Let $k\leq n$ be positive integers and $\groundset{n}$ be the set of integers modulo $n$. A conjecture of Baranyai from 1974 asks for a decomposition of $k$-element subsets of $\groundset{n}$ into particular families of sets called ``wreaths''. We approach this conjecture from a new algebraic angle by introducing the key object of this paper, the $(n,k)$-wreath matrix $M=M(n,k)$. As our first result, we establish that Baranyai's conjecture is equivalent to the existence of a particular vector in the kernel of $M$. We then employ results from representation theory to study $M$ and its spectrum in detail. In particular, we find all eigenvalues of $M$ and their multiplicities, and identify several families of vectors which lie in the kernel of $M$.  
\end{abstract}

\section{Introduction}
One of the foundational results in the theory of hypergraph decomposition is the following one by Zsolt Baranyai from 1974 \cite{baranyai1974factrization}:

\begin{theorem}\label{thm:hypergraphdecomposition}
Let $n$ and $k$ be two positive integers such that $k\mid n$. Then there exists a partition of the hyperedges of a complete $k$-uniform hypergraph on $n$ vertices into perfect matchings (that is, sets of hyperedges such that each vertex lies in exactly one of these hyperedges). 
\end{theorem}

This result finally resolved whether the necessary divisibility condition for hypergraph decompositions into perfect matchings is also a sufficient one --- a natural generalisation of Kirkman's Schoolgirl problem from 1847 \cite{kirkman1847problem}.

At the end of his paper \cite{baranyai1974factrization}, Baranyai conjectures a further generalization of this theorem. To formulate it, we first adapt a notation due to 
Katona \cite{katona1991renyi}.
(This notation differs to the one of Baranyai, who formulated the conjecture in terms of ``staircases''.)

For two positive integers $k \leq n$, let $\groundset{n}$ be the set of integers modulo $n$ and write $\groundsubs{n}{k}$ for the set of subsets of $\groundset{n}$ of size $k$. Consider a permutation of $\groundset{n}$, that is, a bijection $\pi: \groundset{n} \to \groundset{n}$. We define $\fly_{n,k,\pi} \subset \groundsubs{n}{k}$, the
\textit{$(n,k,\pi)$-wreath}
, as 
\[\{\{\pi((i-1)k+1), \pi((i-1)k+2), \ldots, \pi(ik)\}\,|\, i\in \groundset{n}\}.\]
Alternatively, $\fly_{n,k,\pi}$ can be constructed in steps as follows. We start with $\fly_{n,k,\pi}=\emptyset$. First, we add into $\fly_{n,k,\pi}$ the set of $k$ elements in the first interval of length $k$ in permutation $\pi$. Then, we iteratively keep adding the set formed by the $k$ consecutive elements in $\pi$, starting right after the last interval. We stop once we are about to add the first interval of length $k$ again. See an example in \Cref{fig:wreaths}. When $n$ and $k$ are clear from the context, we simply write $\fly_{\pi}$ instead of $\fly_{n,k,\pi}$.

\begin{figure}[h]\centering
    			\includegraphics[height=5.5 cm]{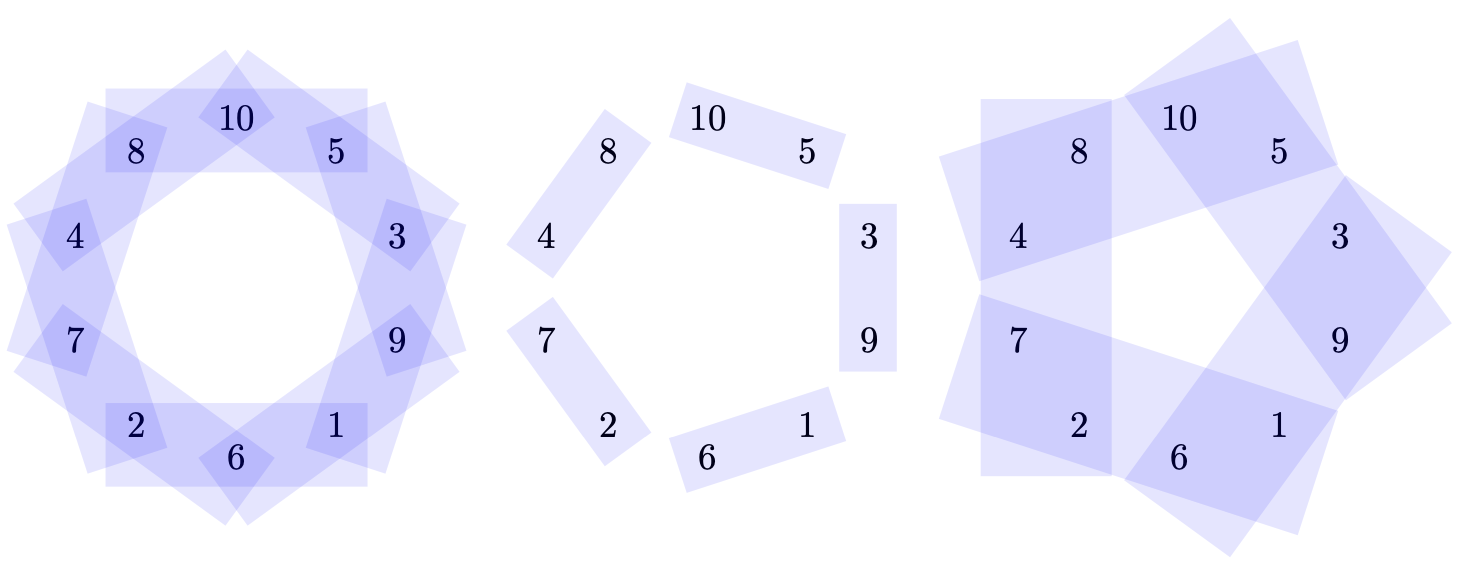}
\caption{The $(n,k,\pi)$-wreaths for $n=10$, $\pi=(10\hquad5\hquad3\hquad9\hquad1\hquad6\hquad2\hquad7\hquad4\hquad8)$ and $k\in \{3,2,4\}$.}
\label{fig:wreaths}
\end{figure}

A set $\fly\subset \groundsubs{n}{k}$ is called an \textit{$(n,k)$-wreath} if it is an $(n,k,\pi)$-wreath for some permutation $\pi$ of $\groundset{n}$, and we write $\mathcal{W}_{n,k}$ for the set of $(n,k)$-wreaths $\fly\subset \groundsubs{n}{k}$. Baranyai's ``wreath conjecture'' \cite{baranyai1974factrization} is as follows:

\begin{conjecture}\label{conj:wreathconj}
It is possible to decompose $\groundsubs{n}{k}$ into disjoint $(n,k)$-wreaths.
\end{conjecture}

If $n$ and $k$ are clear from context, we will shorten ``$(n,k)$-wreath'' into ``\emph{wreath}''.  One parameter influencing the properties of $(n,k)$-wreaths is the greatest common divisor of $n$ and $k$. From the point of computations that appear later in the paper, the case $k\mid n$ sometimes needs to be dealt with separately because such wreaths do not contain overlapping sets of size $k$. This case is also significant for the following reason.

\begin{remark}\label{re:k divides n}
\Cref{thm:hypergraphdecomposition} is a corollary of \Cref{conj:wreathconj} for $k\mid n$.
\end{remark}

Baranyai's proof of \Cref{thm:hypergraphdecomposition} is non-constructive. It follows from a more general result proved by induction on $n$ via the integrity theorem of network flows used in each step. The additional structure of wreaths, however, resists direct generalisations of the proof.

In this paper, we instead approach the wreath conjecture from a new angle. We do so by introducing and studying the \textit{$(n,k)$-wreath matrix} $M=M(n,k)$ which is a square matrix with columns and rows labelled by wreaths from $\mathcal{W}_{n,k}$, such that $M_{\fly, \fly'} = |\fly\cap \fly'|$. Once again, if $n$ and $k$ are clear from context, we shorten ``$(n,k)$-wreath matrix'' into ``wreath matrix''. First, we show that the $(n,k)$-wreath matrix is closely related to \Cref{conj:wreathconj}.

\begin{Prop}\label{prop:tfae}
    Let $k\leq n$ be two integers and let $g=\gcd(n,k)$. The following are equivalent.
    \begin{enumerate}[(a)]
        \item \Cref{conj:wreathconj} holds for $n$ and $k$.
        \item There is $v\in\ker M$ whose entries take only two distinct values, one of which is attained $\frac{g\binom{n}{k}}{n}$-times.
        \item If $k \nmid n$, there is $v \in\ker M$ with at most $\frac{g\binom{n}{k}}{n}$ positive entries and less than $\frac{(n-k)!k!}{2(g!)^{n/g}}$ zeros.

        If $k\mid n$, there is $v \in\ker M$ with at most ${\binom{n-1}{k-1}}$ positive entries and less than $\frac{n(n-k)!}{k(k!)^{n/k-1}(\frac{n}{k})!}$ zeros.
        \item There is $v\in\ker M$ with at most $\frac{g\binom{n}{k}}{n}$ positive entries and such that each subset of $\groundset{n}$ of size $k$ is contained in some wreath $\fly$ which labels a negative entry of $v$.
    \end{enumerate}
\end{Prop}

Our main tools used to analyse the wreath matrix are results from the representation theory of the symmetric groups. Such results have already proved fruitful in studying transitive graphs; see, for instance, \cite[Section~13]{GodsilMeagherEKR15} which presents a machinery for finding the spectrum of several transitive graphs including the Kneser graph. Note that this machinery does not apply to our setting as, in general, the action of the symmetric group on the set of wreaths is not multiplicity-free.

The spectrum of the wreath matrix is described in the main result of this paper.

\begin{theorem}\label{thm:matrixresults}
    Let $k<\frac{n}{2}$ be two positive integers and let $g=\gcd(n,k)$. The $(n,k)$-wreath matrix has $k+1$ eigenvalues $$\lambda_1 \geq \lambda_2,\lambda_3, \ldots, \lambda_k >\lambda_0=0,$$
    all of which are integers. The multiplicity of $\lambda_1$ is $1$. The multiplicity of $\lambda_l$ is $\binom{n}{l}-\binom{n}{l-1}$ for $2 \leq l \leq k$. The value of $\lambda_1$ is $\frac{n (n-k)!k!}{2g(g!)^{n/g}}$ if $k\nmid n$ and $\frac{n^2(n-k)!}{k^2(k!)^{n/k-1}(\frac{n}{k})!}$ if $k\mid n$.
    If $g=1$, for $2\leq l \leq k$ we moreover have 
    \[
    \lambda_l = (-1)^l \frac{1}{2}l! k! (n-k-l)! (n-2k-1)\binom{k}{l} -\frac{1}{2}k! (n-k)! + l! \binom{k+1}{l+1}\sum_{j=0}^l (-1)^{l-j} (k-j)!(n-k-l+j)!.
    \]
\end{theorem}

We stress that the theorem does not say that the $k+1$ eigenvalues are pairwise distinct (see \Cref{qu:distinct}). If some of them are equal, then the multiplicity of an eigenvalue $\lambda$ simply equals the sum of the multiplicities of $\lambda_l$ from the theorem such that $\lambda_l = \lambda$.

We remark that the assumption $k<\frac{n}{2}$ is not restrictive. Indeed, if $k>\frac{n}{2}$, then the bijection between $\mathcal{W}_{n,k}$ and $\mathcal{W}_{n,n-k}$ that replaces all sets in a wreath by their complements preserves the wreath matrix. And for $k=\frac{n}{2}$ wreaths are pairs of complement sets; hence in this case the wreath matrix is twice the identity matrix. The bijection above and the simple structure of wreaths for $k=\frac{n}{2}$ also allow us to restrict to $k<\frac{n}{2}$ in \Cref{conj:wreathconj}.

In fact, our results provide closed formulas for the eigenvalues of the wreath matrix for all $g$. They can be obtained from \Cref{le:lambdaformula}, \Cref{pr:eig_k_div_n} and the formulas in \Cref{ap:coefficients}.
For fixed $k$ and $l$ the eigenvalue $\lambda_l$ divided by $(n-k-l)!$ is a polynomial in $n$ provided $k$ and $n$ are coprime. In \Cref{pr:root} we see that this polynomial has always at least one integral root. Based on these surprising results, the authors wonder if there are further significances of these polynomials.

In light of \Cref{re:k divides n}, the hardest case of \Cref{conj:wreathconj} seems to be $g=1$. It matches the respective case of Bailey--Stevens conjecture about decompositions of complete $k$-uniform hypergraphs into tight Hamiltonian cycles \cite{BAILEY20103088}. Despite the efforts of multiple researchers \cite{BAILEY20103088,huo2015jirimutu, meszka2009decomposing, turek2024intervals, SAENGCHAMPA2024114197, xu2002hamiltonian}, both the wreath conjecture and Bailey--Stevens conjecture are still widely open. Even for $k=3$, they have been confirmed only for a few families of positive integers $n$. We conclude the brief historical overview by mentioning some of the extensions of \Cref{thm:hypergraphdecomposition} in other directions: \cite{bahmanian2014connected, he2022non, katona2023baranyai, xi2024variants}.

The paper is organised as follows. In \Cref{se:matrix} we present the motivation behind the introduction of the wreath matrix, establish its basic properties and prove \Cref{prop:tfae}. In \Cref{se:reptheory} we introduce the necessary background from the representation theory of the symmetric groups and apply it to find multiplicities of the eigenvalues of the wreath matrix. We compute these eigenvalues and prove \Cref{thm:matrixresults} in \Cref{se:evalues}. Finally, we analyse the kernel of the wreath matrix in \Cref{se:kernel}.

\section{The wreath matrix}\label{se:matrix}

Throughout the rest of the paper $n$ and $k$ are positive integers and $g=\gcd(n,k)$. From the definition of $(n,k)$-wreaths, we see that each wreath in $\mathcal{W}_{n,k}$ contains $n/g$ sets. Thus the collection of wreaths required in \Cref{conj:wreathconj} has size $\frac{g\binom{n}{k}}{n}$. 

We can rephrase \Cref{conj:wreathconj} in terms of the largest independent set of a well-chosen graph $G_{\mathcal{W}_{n,k}}$. The vertex set of $G_{\mathcal{W}_{n,k}}$ is $\mathcal{W}_{n,k}$ and we put an edge between two distinct wreaths if they are not disjoint.

Then \Cref{conj:wreathconj} is true if and only if there is an independent set of size $\frac{g\binom{n}{k}}{n}$ in $G_{\mathcal{W}_{n,k}}$. Indeed, by a counting argument, a set of $\frac{g\binom{n}{k}}{n}$ wreaths partitions $\groundsubs{n}{k}$ if and only if its wreaths are pairwise disjoint, that is they form an independent set in $G_{\mathcal{W}_{n,k}}$. Moreover, such an independent set would be the largest, again, by a counting argument.

The introduction of the wreath matrix is motivated by the Delsarte--Hoffman upper bound for the size of an independent set. The Delsarte--Hoffman bound has been widely used in extremal combinatorics, for instance, in Lovász's proof \cite{lovasz1979shannon} of the Erd\H{o}s--Ko--Rado theorem \cite{erdos1961intersection}. To state it we need the following definition. 

By a \emph{pseudoadjacency matrix of a graph $G$} we mean any real matrix $\widetilde{M}$ indexed by $V(G)$ which
\begin{itemize}
\item is symmetric, and 
\item satisfies $\widetilde{M}_{u,v}=0$ whenever $\{u,v\} \notin E(G)$, and
\item whose row-sums are equal and positive (and therefore $f=(1,\ldots,1)$ is an eigenvector).
\end{itemize}
The Delsarte--Hoffman bound (as stated, e.g., in \cite{ellis2021intersection}) is as follows:

\begin{theorem}\label{thm:DelsarteHoffman}
Let $G$ be a graph on $N$ vertices and let $\widetilde{M}$ be a pseudoadjacency matrix of $G$. Let $\lambda_1$ be the eigenvalue
of $\widetilde{M}$ corresponding to the eigenvector $f= (1, 1, \ldots, 1)$, and let $\lambda_{\min}$ be the minimum eigenvalue of $\widetilde{M}$. Let $S \subset V$ be an independent
set of vertices of $G$. Then

$$|S|\leq \frac{-\lambda_{\min}}{\lambda_{1}-\lambda_{\min}}N.$$

Equality holds if and only if

$$1_{S}-\frac{|S|}{N}f$$

is a $\lambda_{\min}$-eigenvector of $\widetilde{M}$, where $1_S$ denotes the indicator vector of $S$.
\end{theorem}

Analogously to the case of adjacency matrices, one gets the following:

\begin{remark}\label{re:max evalue}
    If all entries of $\widetilde{M}$ are non-negative, the maximum eigenvalue of $\widetilde{M}$ is $\lambda_1$.
\end{remark}

For our graph $G_{\mathcal{W}_{n,k}}$, we can obtain a pseudoadjacency matrix $\widetilde{M}$ from the wreath matrix $M$ by setting $\widetilde{M} = M - \frac{n}{g}I$. As $G_{\mathcal{W}_{n,k}}$ is a transitive graph, the row-sums of $\widetilde{M}$ are equal, and since it is connected and $\widetilde{M}$ is non-negative, the row-sums are positive. Then it is easy to see that $\widetilde{M}$ is indeed a pseudoadjacency matrix of $G_{\mathcal{W}_{n,k}}$.

Inspired by \Cref{thm:DelsarteHoffman} (and \Cref{prop:tfae}) we would like to understand the spectrum of $\widetilde{M}$. As $M$ and $\widetilde{M}$ differ only by a scalar matrix, we can study the spectrum of $M$ instead. We start by computing the largest eigenvalue $\lambda_1$ of $M$.

To achieve this, we introduce an action of the symmetric group on $\mathcal{W}_{n,k}$. Throughout the paper we write $S_n$ for the symmetric group of permutations of $\groundset{n}$ with its identity denoted as $\id$. It is easy to see that there is a well-defined action of $S_n$ on $\mathcal{W}_{n,k}$ given by $\tau \fly_{\pi} = \fly_{\tau\pi}$ for any permutations $\tau$ and $\pi$ in $S_n$. The stabiliser of $\fly_{\id}$ consist of permutations which preserve the set partition $\groundset{n}=\cup_{i=1}^{\frac{n}{g}} P_i$, where $P_i=\{ (i-1)g + 1, (i-1)g + 2, \dots, ig \}$, and act on the indices of $P_i$ as elements of the dihedral group $D_{\frac{2n}{g}}$ if $k\nmid n$ or as elements of the symmetric group $S_{\frac{n}{k}}$ if $k \mid n$. Thus, if $k \nmid n$, the stabiliser has size $\frac{2n(g!)^{n/g}}{g}$ and is isomorphic to the \emph{wreath product} $S_g\wr D_{\frac{2n}{g}}$. If $k \mid n$, the stabiliser has size $(k!)^{n/k}(\frac{n}{k})!$ and is isomorphic to $S_k\wr S_{\frac{n}{k}}$.
By the orbit-stabiliser theorem, the size of $\mathcal{W}_{n,k}$ is $\frac{g(n-1)!}{2(g!)^{n/g}}$ if $k \nmid n$ or $\frac{n!}{(k!)^{n/k}(\frac{n}{k})!}$ if $k \mid n$.

\begin{lemma}\label{le:lambda1}
    The largest eigenvalue of $M$ is $\frac{n (n-k)!k!}{2g(g!)^{n/g}}$ when $k \nmid n$ or $\frac{n^2(n-k)!}{k^2(k!)^{n/k-1}(\frac{n}{k})!}$ when $k \mid n$.
\end{lemma}

\begin{proof}
    Let $T\in\groundsubs{n}{k}$. There are $\frac{n (n-k)!k!}{g}$ permutations $\pi\in S_n$ such that $T\in \fly_{\pi}$. Now, assume $k \nmid n$. Using the above size of the stabiliser of $\fly_{\id}$, we conclude that there are $\frac{(n-k)!k!}{2(g!)^{n/g}}$ wreaths which contain $T$. The largest eigenvalue equals the sum of entries of any row of $M$. The sum of a row labelled by $\fly$ is
    \[
    \sum_{\fly'\in\mathcal{W}_{n,k}} |\fly\cap \fly'| = \sum_{T\in \fly}\sum_{\substack{\fly'\in \mathcal{W}_{n,k} \\ T\in \fly'}} 1 = \sum_{T\in \fly} \frac{(n-k)!k!}{2(g!)^{n/g}} = \frac{n}{g}\frac{(n-k)!k!}{2(g!)^{n/g}},
    \]
    as required.
    
    If $k\mid n$, there are $\frac{n(n-k)!}{k(k!)^{n/k-1}(\frac{n}{k})!}$ wreaths which contain $T$. The largest eigenvalue is therefore
    \[
    \sum_{\fly'\in\mathcal{W}_{n,k}} |\fly\cap \fly'| = \sum_{T\in \fly}\sum_{\substack{\fly'\in \mathcal{W}_{n,k} \\ T\in \fly'}} 1 = \sum_{T\in \fly} \frac{n(n-k)!}{k(k!)^{n/k-1}(\frac{n}{k})!} = \frac{n^2(n-k)!}{k^2(k!)^{n/k-1}(\frac{n}{k})!},
    \]
    as required.
\end{proof}

For any $T\in\groundsubs{n}{k}$ we write $w_T$ for the Boolean vector indexed by wreaths, which has ones at entries given by wreaths containing $T$. We also define a Boolean matrix $M_T=w_T w_T^T$. We directly see that $M_T$ are rank one symmetric positively semi-definite matrices. Moreover, $M$ equals the sum of matrices $M_T$ where $T$ runs over the subsets of $\groundset{n}$ of size $k$. We immediately deduce the following.

\begin{Prop}\label{pr:positive}
The matrix $M$ is positively semi-definite.
\end{Prop}

We can conclude a bit more from the decomposition $M = \sum_{T \in \groundsubs{n}{k}} M_T$. In particular, we see that the image of $M$ is spanned by vectors $w_T$ and therefore has dimension at most $\binom{n}{k}$. Moreover, a vector $v$ lies in the kernel of $M$ if and only if it is orthogonal to all vectors $w_T$. We end the section by proving \Cref{prop:tfae} which brings \Cref{conj:wreathconj} and the wreath matrix together.

\begin{proof}[Proof of \Cref{prop:tfae}]
    We prove the statement fully only for $k \nmid n$. The case $k \mid n$ is proved analogously; the difference in equivalent condition (c) comes from a different form of the size of the stabiliser.
    
    Let $c= \frac{g\binom{n}{k}}{n}$ be the desired number of wreaths in \Cref{conj:wreathconj}. Recall that $\widetilde{M} = M - \frac{n}{g} I$ is a pseudoadjacency matrix of $G_{\mathcal{W}_{n,k}}$. Suppose that \Cref{conj:wreathconj} holds for given $k\leq n$ and let $S$ be the desired set of wreaths. In the setting of \Cref{thm:DelsarteHoffman} applied with $G_{\mathcal{W}_{n,k}}$ and $\widetilde{M}$, we have $N=\frac{g(n-1)!}{2(g!)^{n/g}}$, $\lambda_{\min} = -\frac{n}{g}$ (as $M$ is positively semi-definite) and $\lambda_1=\frac{n (n-k)!k!}{2g (g!)^{n/g}} + \lambda_{\min}$ (by \Cref{le:lambda1}).
    
    One readily checks that $|S|=c=\frac{-\lambda_{\min}}{\lambda_{1}-\lambda_{\min}}N$. We thus obtain equality, which provides us with a vector in the kernel of $M$ which clearly satisfies the conditions in (b)\footnote{Alternatively, one can verify that the vector from \Cref{thm:DelsarteHoffman} satisfies the conditions in (b) directly without using the theorem itself.}. Thus (a) $\Rightarrow$ (b).
    
    Any vector in the kernel of $M$ is orthogonal to $\sum_{T\in\groundsubs{n}{k}} w_T$, which has all entries equal. Therefore the vector in (b) must have (up to multiplication by scalar $-1$) $c$ positive entries and the remaining entries are negative. Thus it satisfies the conditions in (c), and so (b) $\Rightarrow$ (c).

    To show (c) $\Rightarrow$ (d), pick vector $v$ as in (c) and $T\in\groundsubs{n}{k}$. Again, $v$ is orthogonal to $w_T$. Either there exists wreath $\fly$ which contains $T$ such that the entry of $v$ indexed by $\fly$ is negative, as desired, or the entries of $v$ indexed by all such wreaths are zero. But there are $\frac{(n-k)!k!}{2(g!)^{n/g}}$ such wreaths, which contradicts the final condition of (c).

    Finally, suppose that (d) holds. Let $P(v)$ be the set of wreaths which index positive entries of $v$. We claim that $P(v)$ is a suitable set of wreaths for \Cref{conj:wreathconj}. As $|P(v)|\leq c$, it is sufficient to show that each $T\in\groundsubs{n}{k}$ lies in at least one wreath in $P(v)$. Similarly to the last paragraph, as $v$ is orthogonal to $w_T$, this would only fail if all entries of $v$ indexed by wreaths which contain $T$ were all zero. But this is forbidden by (d).
\end{proof}

\section{Action of the symmetric group}\label{se:reptheory}
In the rest of the paper we assume $k<\frac{n}{2}$. As justified after \Cref{thm:matrixresults}, this assumption is not restrictive.

To analyse matrix $M$ we apply results from the representation theory of the symmetric groups. The necessary background is provided below; more details can be found, for example, in James \cite{JamesSymmetric78}.

Recall that we write $S_n$ for the symmetric group of permutations of $\groundset{n}$. The irreducible $\C S_n$-modules are given by \textit{Specht modules} $S^{\lambda}$ which are labeled by partitions $\lambda$ of $n$. For $l\leq k$ let $V_l$ be the permutation module of $\C S_n$ with basis $\{ v_A : A \in \groundsubs{n}{l}\}$ equipped with the natural action of $S_n$ on subsets of $\groundset{n}$ of size $l$. We remark that module $V_l$ is isomorphic to the permutation module commonly denoted as $M^{(n-l,l)}$.

If $l\leq l'$, there is a natural inclusion of $V_l$ in $V_{l'}$ given by $v_A \mapsto \frac{1}{\binom{n-l}{l'-l}}\sum_{A \subseteq B} v_B$ (where $|A|=l$ and $|B|=l'$); see \cite[Exercise~7.75b]{StanleyEnumerativeII99}. From now on we will abuse the notation and for $0\leq l\leq k$ identify module $V_l$ with its image in $V_k$.

By \cite[Example~17.17]{JamesSymmetric78} it follows that for any $l\leq k$ there is a decomposition
\[
V_l = \bigoplus_{i=0}^l U_i,
\]
where $U_i\cong S^{(n-i,i)}$. Consequently, for $l \geq 1$, we have decomposition $V_l= V_{l-1} \oplus U_l$, from which it follows that $U_l$ has dimension $\binom{n}{l} - \binom{n}{l-1}$.

Using the construction of Specht modules \cite[Section~4]{JamesSymmetric78}, if $[l]_n = \{ 1,\dots , l\}\subseteq \groundset{n}$, the submodule $U_l$ is generated by
\begin{equation}\label{eq:generator}
    t_l:=\prod_{i=1}^l (\id - \tau_{i, l+i}) v_{[l]_n} = \sum_{\pi \in C_2^l} \sgn(\pi) v_{\pi [l]_n},
\end{equation}
where $\tau_{i, l+i}$ is the transposition $(i\hquad l+i)$ and $C_2^l$ is a subgroup of $S_n$ generated by such transpositions with $1\leq i\leq l$.

We can now return to wreaths. Let $V$ be a complex vector space with a basis $\{e_\fly : \fly \in\mathcal{W}_{n,k} \}$. The action of $S_n$ on the set of wreaths makes $V$ into a permutation module of $\C S_n$. We define a $\C S_n$-module homomorphism $\Phi: V_k\to V$ given by $\Phi(v_T) = w_T$ for any $T\in \groundsubs{n}{k}$, where $w_T = \sum_{T\in \fly} e_\fly$. We use the decomposition $V_k = \bigoplus_{l=0}^k U_l$, to understand $\Phi$. The following is a key computational lemma.

\begin{lemma}\label{le:nonzero}
    Let $0\leq l \leq k$ be a positive integer with $l\neq 1$. Then $\Phi(t_l)\neq 0$.
\end{lemma}

\begin{proof}
    If $l=0$, then $\Phi(t_0)=\frac{1}{\binom{n}{k}}\sum_{|B|=k} w_B\neq 0$. Suppose now that $l\geq 2$.
    
    We expand $\binom{n-l}{k-l}\Phi(t_l)$ as $\sum_\fly \alpha_\fly e_\fly$ and let $\fly'$ be the wreath in \Cref{fig:contribution}. We claim that $|\alpha_{\fly'}| =1$, which implies the result. Using the embedding of $V_l$ into $V_k$, we compute
    
    \begin{align*}
        \binom{n-l}{k-l}\Phi(t_l) &= \binom{n-l}{k-l} \sum_{\pi \in C_2^l} \sgn(\pi) \Phi(v_{\pi [l]_n})\\
        &= \sum_{\pi \in C_2^l} \sgn(\pi) \sum_{\substack{[l]_n\subseteq B \\ |B|=k }} \Phi(v_{\pi B}) \\
        &= \sum_{\pi \in C_2^l} \sum_{\substack{[l]_n\subseteq B \\ |B|=k }} \sgn(\pi) w_{\pi B}\\
        &= \frac{1}{2}\sum_{\pi \in C_2^l} \sum_{\substack{[l]_n\subseteq B \\ |B|=k }} \sgn(\pi) (w_{\pi B} - w_{\pi\tau_{1, l+1} B})\\
        &= \frac{1}{2}\sum_{\pi \in C_2^l} \sum_{\substack{[l]_n\subseteq B \\ |B|=k \\ l+1\notin B }} \sgn(\pi) (w_{\pi B} - w_{\pi\tau_{1, l+1} B})\\
        &= \sum_{\pi \in C_2^l} \sum_{\substack{[l]_n\subseteq B \\ |B|=k \\ l+1\notin B}} \sgn(\pi) w_{\pi B}.
    \end{align*}
    For any $\pi \in C_2^l$ and any set $B$ of size $k$ which contains $[l]_n$ as a subset but does not contain $l+1$, exactly one of $1$ and $l+1$ lies in $\pi B$. Focusing on $\fly'$, and in particular, the fact that $1$ and $l+1$ lie next to each other, we get contribution to $\alpha_{\fly'}$ from $w_{\pi B}$ only if $\pi B$ is either first $k$ clockwise elements of $\fly'$ starting with $l+1$ or last $k$ clockwise elements of $\fly'$ ending with $1$. However, the latter cannot happen as in such case $\pi B$ does not contain $2$ or $l+2$ (otherwise, its size would be at least $n-k>k$), which cannot happen for $l\geq 2$. Thus we get contribution only for $\pi = \prod_{i=1}^l \tau_{i, l+i}$ and $B=[l]_n\cup \{2l+1,2l+2,\dots, l+k\}$. Thus $|\alpha_{\fly'}|=1$ as claimed.
\end{proof}

\begin{figure}[h]
    \centering
    \begin{tikzpicture}
    \draw[rotate=0*360/22](0,4) node {$l+1$};
    \draw[rotate=21*360/22](0,4) node {$l+2$};
    \draw[rotate=20*360/22](0,4) node {$l+3$};
    \draw[rotate=19*360/22](0,4) node {$.$};
    \draw[rotate=18*360/22](0,4) node {$.$};
    \draw[rotate=17*360/22](0,4) node {$.$};
    \draw[rotate=16*360/22](0,4) node {$l+k$};
    \draw[rotate=15*360/22](0,4) node {$2$};
    \draw[rotate=14*360/22](0,4) node {$3$};
    \draw[rotate=13*360/22](0,4) node {$4$};
    \draw[rotate=12*360/22](0,4) node {$.$};
    \draw[rotate=11*360/22](0,4) node {$.$};
    \draw[rotate=10*360/22](0,4) node {$.$};
    \draw[rotate=9*360/22](0,4) node {$l$};
    \draw[rotate=8*360/22](0,4) node {$l+k+1$};
    \draw[rotate=7*360/22](0,4) node {$l+k+2$};
    \draw[rotate=6*360/22](0,4) node {$l+k+3$};
    \draw[rotate=5*360/22](0,4) node {$.$};
    \draw[rotate=4*360/22](0,4) node {$.$};
    \draw[rotate=3*360/22](0,4) node {$.$};
    \draw[rotate=2*360/22](0,4) node {$n$};
    \draw[rotate=1*360/22](0,4) node {$1$};
    \end{tikzpicture}
    \caption{Wreath $\fly' = \fly_{\pi'}$ used in the proof of \Cref{le:nonzero}. Here we take $\pi'(1) = l +1 $.}
    \label{fig:contribution}
\end{figure}
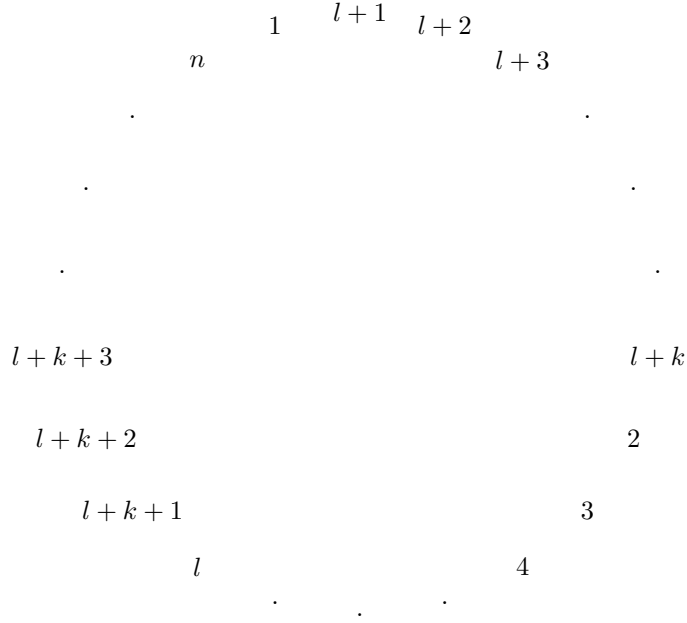

We are now ready to describe the image and kernel of $\Phi$.

\begin{lemma}\label{le:modhom}
    The image of $\Phi$ is the vector space $W\leq V$ spanned by vectors $w_T$ (with $|T|=k$). Moreover, the kernel of $\Phi$ is isomorphic to $U_1$.
\end{lemma}

\begin{proof}
    The first statement is trivial from the definition of $\Phi$. To prove the `moreover' part we firstly need to show that $U_1$ lies in the kernel of $\Phi$, that is $\Phi(V_1) = \Phi(V_0)$. In particular, we can show that $\Phi(v_{\{1\}})$ equals $\Phi(v_{\emptyset})$. But this is clear as
    \[
    \Phi(v_{\{1\}}) = \frac{1}{\binom{n-1}{k-1}}\sum_{\substack{1\in T \\ |T|=k}} w_T =\frac{1}{n\binom{n-1}{k-1}}\sum_{i=1}^n\sum_{\substack{i\in T \\ |T|=k}} w_T =
    \frac{k}{n\binom{n-1}{k-1}} \sum_{|T|=k} w_T = \Phi(v_{\emptyset}).
    \]
Finally, the kernel cannot be any larger as it is a submodule of $V_k$ and does not contain $U_l$ for $l\neq 1$ by \Cref{le:nonzero}.
\end{proof}

For readers with further algebraic interests, we include the algebraic interpretation of the map $\Phi$.

\begin{remark}
    Let $G=S_n$, $H=S_k\times S_{n-k}$ and $K$ be the wreath product $S_g\wr D_{\frac{2n}{g}}$ if $k\nmid n$ or $S_k\wr S_{\frac{n}{k}}$ if $k\mid n$. If we write $\C$ for the trivial module, we obtain isomorphism $V_k\cong \C\big\uparrow _{H}^{G}$ and $V\cong \C\big\uparrow _{K}^{G}$. In turn $\Phi$ lies in $\Hom_{G}(\C\big\uparrow _{H}^{G},\C\big\uparrow _{K}^{G})$. Frobenius reciprocity, Mackey theorem and Schur's lemma give rise to isomorphisms $\Hom_{G}(\C\big\uparrow _{H}^{G},\C\big\uparrow _{K}^{G})\cong \Hom_{H}(\C,\C\big\uparrow _{K}^{G}\big\downarrow _{H}^{G}) \cong \bigoplus_{t\in H\backslash G / K} \Hom_{H}(\C,\C\big\uparrow_{H\cap tKt^{-1}}^{H}) \cong \bigoplus_{t\in H\backslash G / K} \Hom_{H\cap tKt^{-1}}(\C,\C)$. This is just a vector space with a basis indexed by double cosets $t\in H\backslash G / K$ and one can verify that under the above isomorphism, the map $\Phi$ corresponds to the basis element labelled by $t=1$. 
\end{remark}

We define $E_i=\Phi(U_i)\cong U_i$ for $i=0$ and $2\leq i\leq k$, $W_l=\Phi(V_l)$ for $0\leq l\leq k$ and $w_A = \binom{n-l}{k-l}\Phi(v_A)$ for any set $A$ of size $l\leq k$, which generates the module $W_l$. The notation $w_A$ is consistent with the early notation for sets $A$ of size $k$. Thus $W=W_k$, $W_l= E_0 \oplus E_2 \oplus E_3 \oplus \ldots \oplus E_l$ and $w_A = \sum_{A\subseteq B} w_B$, where the sum runs over sets $B$ of size $k$. For $i=0$ we can see that $E_0$ is in fact the one-dimensional module spanned by vector $\sum_{\fly\in \mathcal{W}_{n,k}} e_{\fly}$. A diagram which shows the relations between the various $\C S_n$-modules is in \Cref{fig:modules}.

\begin{figure}[h]
    \centering
    \begin{tikzpicture}
	\node (1) at (-14.35, 7) {};
	\node (2) at (-7.75, 7) {};
	\node (3) at (-11, 7.5) {$V_i$};
	\node (8) at (-9.5, 5.25) {};
	\node (9) at (-9.5, 3.75) {};
	\node (10) at (-10, 4.5) {$\Phi$};
	\node (11) at (-14, 6.25) {$U_0$};
	\node (12) at (-13.25, 6.25) {$\oplus$};
	\node (13) at (-12.5, 6.25) {$U_1$};
	\node (14) at (-11.75, 6.25) {$\oplus$};
	\node (15) at (-11, 6.25) {$U_2$};
	\node (16) at (-10.25, 6.25) {$\oplus$};
	\node (17) at (-9.5, 6.25) {$\cdots$};
	\node (18) at (-8.75, 6.25) {$\oplus$};
	\node (19) at (-8, 6.25) {$U_i$};
	\node (20) at (-7.25, 6.25) {$\oplus$};
	\node (21) at (-6.5, 6.25) {$\cdots$};
	\node (22) at (-5.75, 6.25) {$\oplus$};
	\node (23) at (-5, 6.25) {$U_k$};
	\node (24) at (-14.35, 2) {};
	\node (25) at (-7.75, 2) {};
	\node (26) at (-11, 1.5) {$W_i$};
	\node (27) at (-14, 2.75) {$E_0$};
	\node (28) at (-13.25, 2.75) {$\oplus$};
	\node (29) at (-12.5, 2.75) {$0$};
	\node (30) at (-11.75, 2.75) {$\oplus$};
	\node (31) at (-11, 2.75) {$E_2$};
	\node (32) at (-10.25, 2.75) {$\oplus$};
	\node (33) at (-9.5, 2.75) {$\cdots$};
	\node (34) at (-8.75, 2.75) {$\oplus$};
	\node (35) at (-8, 2.75) {$E_i$};
	\node (36) at (-7.25, 2.75) {$\oplus$};
	\node (37) at (-6.5, 2.75) {$\cdots$};
	\node (38) at (-5.75, 2.75) {$\oplus$};
	\node (39) at (-5, 2.75) {$E_k$};
	\draw (1.center) to (2.center);
	\draw [->] (8.center) to (9.center);
	\draw (24.center) to (25.center);
    \end{tikzpicture}
    \caption{The relations between the introduced $\C S_n$-modules.}
    \label{fig:modules}
\end{figure}
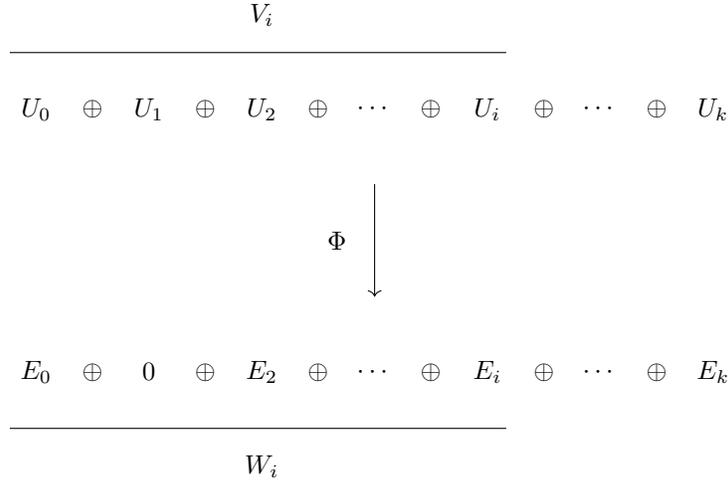

We end the section with the first result about the eigenvalues of $M$. The key ingredient is Schur's lemma \cite[Proposition~4]{SerreRepresentations77}, which implies that any endomorphism of $W= E_0 \oplus E_2 \oplus E_3 \oplus \ldots \oplus E_k$ acts on each $E_i$ by scalar multiplication. As in \Cref{thm:matrixresults}, the eigenvalues may not be pairwise distinct.

    \begin{Prop}\label{pr:multiplicities}
The matrix $M$ has $k+1$ eigenvalues $\lambda_1 \geq \lambda_2,\lambda_3, \ldots, \lambda_k >\lambda_0=0$. The multiplicity of $\lambda_1$ is $1$. The multiplicity of $\lambda_l$ is $\binom{n}{l}-\binom{n}{l-1}$ for $2 \leq l \leq k$.
\end{Prop}

\begin{proof}
Let $\alpha$ be the linear endomorphism of $V$ given by $M$. Clearly, $\alpha$ commutes with the action of the symmetric group, in other words, $\alpha$ is an endomorphism of the $\C S_n$-module $V$. Moreover, from the decomposition $M=\sum_{T\in\groundsubs{n}{k}} M_T$, we know that $\alpha$ has image $W=\Phi(V_k)$ and we can thus restrict $\alpha$ to $W$, which yields a positive definite map. By Schur's lemma, $\alpha$ acts on each submodule $E_l$ by scalar multiplication (where $l=0$ or $2\leq l\leq k$). We know that the one-dimensional space $E_0$ is the eigenspace of the largest eigenvalue $\lambda_1$. And for $2\leq l\leq k$ if $\lambda_l$ is the scalar by which $\alpha$ acts on $E_l$, then $\lambda_l>0$ and its multiplicity is the dimension of $E_l$ which is $\binom{n}{l}-\binom{n}{l-1}$, as required.
\end{proof}

\section{Computing eigenvalues}\label{se:evalues}

We now compute the eigenvalues $\lambda_l$ of $M$ from \Cref{pr:multiplicities} for $2\leq l\leq k$. As noted at the end of the previous section, the eigenspace corresponding to $\lambda_l$ is $E_l$.

\begin{lemma}\label{le:inclusionexclusion}
    Let $A$ be a subset of $\groundset{n}$ of size $l$ such that $2\leq l\leq k$.
    \begin{enumerate}
        \item For a set $J\subseteq A$ we have
        \[
        \sum_{\substack{T\cap A = J \\ |T| = k}} w_T \equiv (-1)^{|A\setminus J|} w_A \mod{W_{l-1}}.
        \]
        \item For a non-negative integer $j\leq l$ we have
        \[
        \sum_{\substack{|T\cap A| = j \\ |T| = k}} w_T \equiv (-1)^{l-j}\binom{l}{j} w_A \mod{W_{l-1}}.
        \]
    \end{enumerate}
\end{lemma}

\begin{proof}
    Using the definition of $v_C\in V_k$ for any set $C$ of size at most $k$ and the principle of inclusion and exclusion, we compute
    \[
    \sum_{\substack{T\cap A = J \\ |T| = k}} v_T = \sum_{J\subseteq C\subseteq A}(-1)^{|C\setminus J|}\sum_{\substack{T\supseteq C \\ |T| = k}} v_T = \sum_{J\subseteq C\subseteq A}(-1)^{|C\setminus J|}\binom{n-|C|}{k-|C|}v_C.
    \]
    The right hand side becomes $(-1)^{|A\setminus J|} \binom{n-l}{k-l} v_A$ modulo $V_{l-1}$ and we obtain the first part after applying $\Phi$. Summing the congruence over all sets $J\subseteq A$ of a fixed size $j$ yields the second part. 
\end{proof}

For $2\leq l\leq k$ and $0\leq j\leq l$ we define the constant $b_{l,j}$ as the inner product of vectors $w_{[l]_n}$ and $w_{\{ l-j+1,l-j+2,\ldots, l-j+k\}}$. By symmetry, we have $\langle w_T,w_A \rangle=b_{l,j}$ whenever $A$ and $T$ are sets of sizes $l$ and $k$, respectively, with an intersection of size $j$. Hence, we compute the following.

\begin{lemma}\label{le:lambdaformula}
    Let $2\leq l\leq k$. Then 
    \[
    \lambda_l = \sum_{j=0}^l (-1)^{l-j} \binom{l}{j} b_{l,j}.
    \]
\end{lemma}

\begin{proof}
    Recall that $M$ acts by multiplication by $\lambda_l$ on $E_l$ and $W_l = W_{l-1}\oplus E_l$. For any set $A$ of size $l$ we have that $w_A$ generates the module $W_l$, and thus $w_A\in W_l\setminus W_{l-1}$. Therefore $\lambda_l$ is the eigenvalue of $M$ acting on $w_A$ modulo $W_{l-1}$. Using \Cref{le:inclusionexclusion} we compute
    \begin{align*}
        M w_A &\equiv \sum_{T\in \groundsubs{n}{k}} M_T w_A \equiv \sum_{T\in \groundsubs{n}{k}} \langle w_T, w_A \rangle w_T \equiv \sum_{j=0}^l b_{l,j} \sum_{\substack{|T\cap A| = j \\ |T| = k}} w_T \\&\equiv \sum_{j=0}^l (-1)^{l-j} \binom{l}{j} b_{l,j} w_A \mod{W_{l-1}},
    \end{align*}
    as desired.
\end{proof}

We can now compute the coefficients $b_{l,j}$ and eigenvalues of $M(n,k)$ when $k \mid n$.

\begin{Prop}\label{pr:eig_k_div_n}
    Let $2\leq l\leq k$ and $0\leq j\leq l$ be integers. Provided that $k$ divides $n$, we have
    \begin{align*}
        b_{l,j} =
        \begin{cases}
        \frac{(n-2k)!}{(k!)^{n/k-2}(\frac{n}{k}-2)!}\binom{n-k-l}{k-l}& \textnormal{if } j=0,\\
        0 & \textnormal{if } 0<j<l,\\
        \frac{(n-k)!}{(k!)^{n/k-1}(\frac{n}{k}-1)!}&\textnormal{if } j=l.\\
        \end{cases}
    \end{align*}
    In particular, we have
    \begin{align*}
        \lambda_l =  \frac{(n-k)!}{(k!)^{n/k-1}(\frac{n}{k}-1)!} + (-1)^l \frac{(n-2k)!}{(k!)^{n/k-2}(\frac{n}{k}-2)!}\binom{n-k-l}{k-l}.
    \end{align*}
\end{Prop}

\begin{proof}
    Let $A = [l]_n$ and $T = \{ l-j+1,l-j+2,\ldots, l-j+k\}$. Recall that $b_{l,j}$ is the inner product of vectors $w_{A}$ and $w_{T}$. Thus it counts pairs $(\fly,S)$, where $\fly$ is a wreath which contains sets $T$ and $S$ (of size $k$) with $A\subseteq S$.
    
    Note that the sets of size $k$ within a wreath are disjoint. Therefore, if $0<j<l$, and hence $|A \cap T|=j>0$ and $|A \setminus T|=l-j>0$, there are no such pairs $(\fly,S)$. We thus have $b_{l,j}=0$ whenever $0<j<l$.

    Assume now $j=l$, that is, $A \subset T$. In this case, we have $S=T$, and therefore $b_{l,l}$ counts the number of wreaths which contain $T$. This is computed in the proof of \Cref{le:lambda1}: as there are $\frac{n (n-k)!k!}{k}$ permutations $\pi\in S_n$ such that $T\in \fly_{\pi}$ and the stabiliser of $\fly_{\id}$ in $S_n$ has size $(k!)^{n/k}(\frac{n}{k})!$, there are $\frac{(n-k)!}{(k!)^{n/k-1}(\frac{n}{k}-1)!}$ wreaths which contain $T$.

    Finally, assume $j=0$. Here, we have $b_{l,0}=\sum_{A\subset S} q_{S,T}$ where $q_{S,T}$ is the number of wreaths which contain $S$ and $T$. If $S$ and $T$ are not disjoint, then $q_{S,T}=0$. If $S$ and $T$ are disjoint, as there are $\frac{n(n-k) (n-2k)!(k!)^2}{k^2}$ permutations $\pi\in S_n$ such that $S,T\in \fly_{\pi}$, we get $q_{S,T}=\frac{(n-2k)!}{(k!)^{n/k-2}(\frac{n}{k}-2)!}$. The formula for $b_{l,0}$ follows from noting that there are $\binom{n-k-l}{k-l}$ ways to choose elements of $S\setminus A$ from $\groundset{n} \setminus \left(A \cup T\right)$.
    
    The value of $\lambda_l$ comes from \Cref{le:lambdaformula}.
\end{proof}

For computing $b_{l,j}$ and the eigenvalues of $M(n,k)$ when $k \nmid n$, the following identity of binomial coefficients is useful.

\begin{lemma}\label{le:binomial}
    Let $\alpha, \beta$ and $\gamma$ be three non-negative integers. Then
    \[
    \sum_{x+y=\gamma} \binom{x}{\alpha}\binom{y}{\beta} = \binom{\gamma+1}{\alpha+\beta+1}. 
    \]
\end{lemma}

\begin{proof}
    Both sides count the number of ways to place $\alpha$ wolves, $\beta$ sheep and a fence between them on $\gamma + 1$ colinear spots. 
\end{proof}

We now present the values of the coefficients $b_{l,j}$ for $n$ and $k$ coprime. Using similar computations one can determine the coefficients even when $n$ and $k$ are not coprime and $k\nmid n$, however, the formulas are rather complicated; see \Cref{ap:coefficients}.

\begin{Prop}\label{pr:bcoeficients}
    Let $2\leq l\leq k$ and $0\leq j\leq l$ be integers. Provided that $n$ and $k$ are coprime, we have
    \begin{align*}
        b_{l,j} =
        \begin{cases}
        \frac{1}{2}l! k! (n-k-l)! \left( 2\binom{k}{l+1} + (n-2k+1) \binom{k}{l}\right) & \textnormal{if } j=0,\\
        j! (l-j)! (k-j)! (n-k-l+j)! \binom{k+1}{l+1} & \textnormal{if } 0<j<l,\\
        \frac{1}{2}l! (k-l)! (n-k)! \left( 2\binom{k}{l+1} + \binom{k}{l}\right) & \textnormal{if } j=l.\\
        \end{cases}
    \end{align*}
\end{Prop}

\begin{proof}
    Let $A = [l]_n$ and $T = \{ l-j+1,l-j+2,\ldots, l-j+k\}$. Recall that $b_{l,j}$ is the inner product of vectors $w_{A}$ and $w_{T}$. Thus it counts pairs $(\fly,S)$, where $\fly$ is a wreath which contains sets $T$ and $S$ (of size $k$) with $A\subseteq S$.
    
    Let $d(\fly)$ be the least integer $d$ such that $A \subseteq \{ \pi(i+1), \pi(i+2), \dots, \pi(i+d) \}$ for some $i\in\groundset{n}$, where $\fly=\fly_{\pi}$. Informally, $d(\fly)$ denotes the least number of consecutive places on $\fly$ which contain $A$. Then for a given $\fly$, there are $k+1-d(\fly)$ choices for $S$ (provided $d(\fly)\leq k$, otherwise there are none). We further let $q_{d,j}$ be the number of wreaths $\fly$ which contain $T$ and have $d(\fly)=d$. We now distinguish the cases $j=0$, $j=l$ and $0<j<l$.

    Suppose firstly that $j=0$. We claim that $q_{d,0}=\frac{1}{2}l! k! (n-k-l)! (n-k-d+1)\binom{d-2}{l-2}$ for $d\leq k$.
    
    To show this, we count the permutations $\sigma \in S_n$ such that $\fly_{\sigma}$ contains $T$ and has $d(\fly_{\sigma})=d$, and then divide this number by $2n$ to account for rotations and reflections.
    Such a permutation $\sigma$ is uniquely given by the following data: the position of the first (clockwise) element of $T$, the (clockwise) number of elements between the last element of $T$ and the first element of $A$, the positions of the elements of $A$ within the $d$ consecutive elements, and the permutations of $T$, $A$ and the remaining elements.
    Vice versa, each valid choice of these parameters uniquely determines $\sigma$.
    
    The number of choices for the position of the first element of $T$ is $n$, for the number of elements between $T$ and $A$ it is $n-k-d+1$, for the placement of $A$ within $d$ elements it is $\binom{d-2}{l-2}$, and the numbers of the final three permutations are $k!$, $l!$ and $(n-k-l)!$, respectively.
    
    After putting all this together, we obtain the claimed formula for $q_{d,0}$. We compute
    \begin{align*}
    2b_{l,0} &= \sum_{d=l}^k (k+1-d) l! k! (n-k-l)!(n-k-d+1)\binom{d-2}{l-2}\\
    &=\sum_{d=l}^k l! k! (n-k-l)!\left( 2\binom{k+1-d}{2}\binom{d-2}{l-2} + (n-2k+1)\binom{k+1-d}{1}\binom{d-2}{l-2} \right)\\
    &=l! k! (n-k-l)! \left( 2\binom{k}{l+1} + (n-2k+1) \binom{k}{l}\right),  
    \end{align*}
    using \Cref{le:binomial} twice with $\alpha=2$ and $1$, $\beta = l-2$ and $\gamma = k-1$.

    If $j=l$, similarly to the case $j=0$, we find that $q_{d,l}=\frac{1}{2}l! (k-l)! (n-k)! (k-d+1)\binom{d-2}{l-2}$. Thus we obtain
    \begin{align*}
    2b_{l,l} &= \sum_{d=l}^k (k+1-d)^2 l! (k-l)! (n-k)!\binom{d-2}{l-2}\\
    &=\sum_{d=l}^k l! (k-l)! (n-k)!\left( 2\binom{k+1-d}{2}\binom{d-2}{l-2} + \binom{k+1-d}{1}\binom{d-2}{l-2} \right)\\
    &=l! (k-l)! (n-k)! \left( 2\binom{k}{l+1} +  \binom{k}{l}\right),  
    \end{align*}
    as desired.

    Finally, for  $0<j<l$, we claim that $q_{d,j}=j! (l-j)! (k-j)! (n-k-l+j)! \binom{d-1}{l-1}$.
    Once again, we do so by computing the number of permutations $\sigma \in S_n$ such that $\fly_{\sigma}$ contains $T$ and has $d(\fly_{\sigma})=d$, and divide this number by $2n$.
    
    This time, $\sigma$ is uniquely determined by the following data: the position of the first (clockwise) element of $T$, the positions of the elements of $A$ so that $d(\fly_{\sigma})=d$ and exactly $j$ of these positions coincide with the positions of $T$, and the permutations of $T \cap A$, $T\setminus A$, $A\setminus T$ and the remaining elements.
    Vice versa, each valid choice of these parameters uniquely determines $\sigma$.
    
    Once more, the number of choices for the position of the first element of $T$ is $n$. The numbers of the final four permutations are $j!$, $(k-j)!$, $(l-j)!$ and $(n-k-l+j)!$, respectively. 
    
    It remains to justify that the remaining term is $2\binom{d-1}{l-1}$. Let $D$ be the set of the $d$ consecutive elements.
    From the first and the last element of $T$, exactly one lies in $D$, which gives the factor of $2$ in $2\binom{d-1}{l-1}$.
    Without loss of generality, assume that it is the last element of $T$ that lies in $D$. Then the position of the first element of $A$ coincides with a position of $T$. Next $j-1$ elements of $A$ still lie in $T$, and the remaining $l-j$ elements of $A$ lie outside of $T$ so that the last one comes $d-1$ positions after the first one. We are now in the setting of \Cref{le:binomial} with $x=|T\cap D|-1$, $y=|D\setminus T|-1$, $\alpha=j-1$ and $\beta=l-j-1$, explaining the factor of $\binom{d-1}{l-1}$.

    After putting all this together, we obtain the claimed formula for $q_{d,j}$. We compute

    \begin{align*}
    b_{l,j} &= \sum_{d=l}^k (k+1-d) j! (l-j)! (k-j)! (n-k-l+j)! \binom{d-1}{l-1}\\
    &=\sum_{d=l}^k j! (l-j)! (k-j)! (n-k-l+j)! \binom{k+1-d}{1}\binom{d-1}{l-1}\\
    &= j! (l-j)! (k-j)! (n-k-l+j)! \binom{k+1}{l+1},  
    \end{align*}
    which finishes the proof.
\end{proof}

Combination of the last two results yields the desired formula for eigenvalues of $M$.

\begin{Cor}\label{co:explicitlambda}
    Let $2\leq l\leq k$ be integers such that $k$ and $n$ are coprime. Then 
    \[
    \lambda_l = (-1)^l \frac{1}{2}l! k! (n-k-l)! (n-2k-1)\binom{k}{l} -\frac{1}{2}k! (n-k)! + l! \binom{k+1}{l+1}\sum_{j=0}^l (-1)^{l-j} (k-j)!(n-k-l+j)!.
    \]
\end{Cor}

\begin{proof}
    This is a combination of \Cref{le:lambdaformula} and \Cref{pr:bcoeficients} after we rewrite
    \begin{align*}b_{l,0}&=\frac{1}{2}l! k! (n-k-l)! \left( 2\binom{k}{l+1} + (n-2k+1) \binom{k}{l}\right)\\&=\frac{1}{2}l! k! (n-k-l)! (n-2k-1)\binom{k}{l} + l! k! (n-k-l)! \binom{k+1}{l+1}
    \end{align*}
    and \begin{align*}b_{l,l}&=\frac{1}{2}l! (k-l)! (n-k)! \left( 2\binom{k}{l+1} + \binom{k}{l}\right)\\&=l! (k-l)! (n-k)! \binom{k+1}{l+1} -\frac{1}{2}k! (n-k)!.
    \end{align*}
\end{proof}

\Cref{thm:matrixresults} follows.

\begin{proof}[Proof of \Cref{thm:matrixresults}]
    This is \Cref{le:lambda1}, \Cref{pr:multiplicities} and \Cref{co:explicitlambda}.
\end{proof}

We finish this section by describing some properties of $\lambda_l$ when $2\leq l \leq k$ for $k$ and $n$ coprime. First, we present explicit formulas for $l=2$ and $l=3$.

\begin{example}\label{ex:evalues}
Assume that $k$ and $n$ are coprime. Then for $k\geq 2$ we have
\[\lambda_2 = \frac{1}{6}k!(n-k-2)!(n-2)\left( (2k-1)n - (3k^2-k-1) \right)\]
and for $k\geq 3$
\[\lambda_3 = \frac{1}{4}k!(n-k-3)!(n-3)(n-2k)\left( (k-1)n - (2k^2-2k-2) \right).\]
\end{example}

While the formula for the eigenvalues of $M$ in \Cref{co:explicitlambda} does not seem to further simplify, there are a couple of interesting observations about the eigenvalues we can make.

\begin{Cor}\label{co:poly}
    Let $2\leq l\leq k$ be integers such that $k$ and $n$ are coprime. Then $\lambda_l=(n-k-l)! P_{k,l}(n)$, where $P_{k,l}$ is a polynomial of degree $l$ with leading term $\frac{(2k-l+1)k!}{2(l+1)}$.
\end{Cor}

\begin{proof}
    From \Cref{pr:bcoeficients}, we see that $b_{l,j}=(n-k-l)! Q_{k,l,j}(n)$ where $Q_{k,l,j}$ is a polynomial of degree $1$ if $j=0$ and degree $j$ otherwise. Moreover, the leading term of $b_{l,l}$ is $\frac{l! (k-l)!}{2} \left( 2\binom{k}{l+1} + \binom{k}{l} \right) = \frac{(2k-l+1)k!}{2(l+1)}$ The result then follows from \Cref{le:lambdaformula}.
\end{proof}

We end the section by establishing some of the roots of $P_{k,l}$.

\begin{Prop}\label{pr:root}
Let $2\leq l\leq k$ be integers such that $k$ and $n$ are coprime.
    \begin{enumerate}[(i)] 
        \item One of the roots of $P_{k,l}$ is $l$.
        \item If $l$ is odd, then $2k$ is a root of $P_{k,l}$.
    \end{enumerate}
\end{Prop}

\begin{proof}
    Using \Cref{co:explicitlambda} we compute that $P_{k,l}(l)$ equals
    \begin{align*}
        P_{k,l}(l) &= (-1)^l \frac{1}{2}l! k! (l-2k-1)\binom{k}{l} -\frac{1}{2}k! (l-k)(l-1-k)\dots (1-k)\\
        &+ l! \binom{k+1}{l+1}\sum_{j=0}^l (-1)^{l-j} (k-j)!(j-k)(j-1-k)\dots (1-k)\\
        &= (-1)^l \frac{1}{2}l! k! (l-2k-1)\binom{k}{l} - (-1)^l\frac{1}{2}k!l! \binom{k-1}{l}+ (-1)^l l!(k-1)! \binom{k+1}{l+1}\sum_{j=0}^l (k-j)\\
        &= (-1)^l\frac{1}{2}(k-1)!l! \left( k(l-2k-1)\binom{k}{l} - k\binom{k-1}{l} + (2k-l)(l+1)\binom{k+1}{l+1} \right)\\
        &=(-1)^l\frac{1}{2}(k-1)!l!\binom{k}{l} \left( k(l-2k-1) - (k-l) + (2k-l)(k+1)\right)\\
        &=(-1)^l\frac{1}{2}(k-1)!l!\binom{k}{l} \left( kl-2k^2-k - k+l + 2k^2-kl +2k -l\right)\\
        &=0,
    \end{align*}
    thus (i) is proven.

    Since $(n-k-l)!$ is defined (and nonzero) for $n=2k$, we can compute that $\lambda_l$ becomes $0$ when we let $n=2k$ in \Cref{co:explicitlambda} (for $l$ odd) to establish (ii). The substitution yields
    
    \begin{align*}
    P_{k,l}(2k)&=(-1)^{l+1} \frac{1}{2}l! k! (k-l)!\binom{k}{l} -\frac{1}{2}(k!)^2 + l! \binom{k+1}{l+1}\sum_{j=0}^l (-1)^{l-j} (k-j)!(k-l+j)!\\
    &=l! \binom{k+1}{l+1}\sum_{j=0}^l (-1)^{l-j} (k-j)!(k-l+j)!.    
    \end{align*}
    The change of variables $i=l-j$ changes the final sum to $$\sum_{i=0}^l (-1)^{i} (k-l+i)!(k-i)! = -\sum_{i=0}^l (-1)^{l-i} (k-i)!(k-l+i)!;$$ thus the sum is zero.
\end{proof}

\section{Kernel of $M$}\label{se:kernel}

Motivated by \Cref{prop:tfae} we would like to understand the kernel of $M$. Using \Cref{pr:multiplicities}, we know that its codimension is $\binom{n}{k} - n +1$. We can also list some elements in the kernel.

Let $a\leq n$ be an integer greater or equal to $2$. We write $\tau_i$ for transposition $(i\hquad i+1)$ and let $x_a = e_{\fly_{\id}} - e_{\fly_{\tau_1}} - e_{\fly_{\tau_{a+1}}} + e_{\fly_{\tau_1 \tau_{a+1}}}$. For any wreath $\fly$ we further define $y_{a,\fly} = \sum_{\sigma \in S_a} \sgn(\sigma) e_{\sigma\fly}$. We can check that vectors $x_a$ and $y_{a,\fly}$ lie in the kernel of $M$ for suitable $a$.

\begin{lemma}\label{le:transpose two}
    Let $a\leq \frac{n}{2}$ be a positive integer distinct from $1$ and $k$. Then $x_a$ lies in the kernel of $M$.
\end{lemma}

\begin{proof}
    We need to show that for any set $T\in \groundsubs{n}{k}$, the vector $w_T$ is orthogonal to $x_a$. Writing $H$ for the subgroup of $S_n$ generated by transpositions $\tau_1$ and $\tau_{a+1}$, we can write $\langle w_T, x_a\rangle = \sum_{\substack{h\in H \\ T\in \fly_{h}}} \sgn(h)$. If $T$ contains an even number of elements of $\{ 1,2\}$ or an even number of elements of $\{ a+1, a+2\}$, then the sum vanishes as the signs cancel out. Otherwise, since $a\neq k$, there is no $h\in H$ such that $T\in \fly_{h}$, and thus we also get zero. 
\end{proof}

For vectors $y_{a,\fly}$ we use a stronger statement which provides a way to generate vectors in the kernel of $M$. To state it, we need some additional terminology. A \textit{linear character} of a group $H$ is a homomorphism $\rho: H\to \C$. We denote the restriction of $\rho$ to a subgroup $K$ of $H$ by $\rho\big\downarrow_K^H$ and the \textit{trivial character} which sends all group elements of $H$ to $1$ by $\mathbf{1}_H$. The well-known \emph{orthogonality of characters} implies that if $\rho$ and $\eta$ are distinct linear characters of $H$, then $\sum_{h\in H} \rho(h)\overline{\eta(h)} = 0$, where the bar denotes the complex conjugate.

Additionally, for subset $T\in\groundsubs{n}{k}$, we write $S_{n,T}$ for the stabiliser of $T$ in $S_n$ (that is, the subgroup of $S_n$ consisting of permutations $\pi$ such that $\pi(T)=T$).

\begin{lemma}\label{le:vectors in kernel}
    Let $H$ be a subgroup of $S_n$. Suppose that $\rho$ is a linear character of $H$ such that for any subset $T\in \groundsubs{n}{k}$ we have $\rho\big\downarrow_{H\cap S_{n,T}}^{H}\neq \mathbf{1}_{H\cap S_{n,T}}$. Then for any wreath $\fly$, the vector $\sum_{h\in H} \rho(h) e_{h\fly}$ lies in the kernel of $M$.
\end{lemma}

\begin{proof}
    Let $T\in\groundsubs{n}{k}$. We need to show that $w_T$ is orthogonal to $\sum_{h\in H} \rho(h) e_{h\fly}$. As mentioned above, the condition $\rho\big\downarrow_{H\cap S_{n,T}}^{H}\neq \mathbf{1}_{H\cap S_{n,T}}$ implies $\sum_{h\in H\cap S_{n,T}} \rho(h)=0$. The inner product $\langle w_T, \sum_{h\in H}\rho(h)e_{h\fly}\rangle$ equals
    \[\sum_{\substack{h\in H \\ T\in h\fly}} \rho(h)= \sum_{S\in \fly} \sum_{\substack{h\in H \\ T= hS}} \rho(h).\]
    If there is no $h'\in H$ such that $T=h'S$, then $\sum_{\substack{h\in H \\ T= hS}} \rho(h)$ is zero. Otherwise, we use such $h'$ to rewrite $\sum_{\substack{h\in H \\ T= hS}} \rho(h)$ as
    \[\sum_{\substack{h\in H \\ hT = T}} \rho(hh') = \rho(h') \sum_{\substack{h\in H \\ hT = T}} \rho(h).\]
   As the stabiliser of $T$ in $H$ is $H\cap S_{n,T}$, the sum is also zero. This finished the proof. 
\end{proof}

For readers with further algebraic interests, we mention a possible reformulation of \Cref{le:vectors in kernel}.

\begin{remark}\label{re:Frobenius}
    By the Frobenius reciprocity, the condition $\rho\big\downarrow_{H\cap S_{n,T}}^{H}\neq \mathbf{1}_{H\cap S_{n,T}}$ can be replaced by: \\ $\rho$ is not an irreducible constituent of the permutation character $\left(\mathbf{1}_{H\cap S_{n,T}}\right)\big\uparrow_{H\cap S_{n,T}}^{H}$.
\end{remark}

We can now deduce that the vectors $y_{a,\fly}$ lie in the kernel of $M$ quite easily.

\begin{Cor}\label{cor:permute set}
    Let $a\leq n$ be a positive integer greater than $2$. If $\fly$ is a wreath, then $y_{a,\fly}$ is in the kernel of $M$.
\end{Cor}

\begin{proof}
    We apply \Cref{le:vectors in kernel} with $H=S_a$ and $\rho = \sgn$. Since $a\geq 3$, any subgroup $S_{n,T}$ with $T\in \groundsubs{n}{k}$ contains a transposition from $H$ and hence $\rho\big\downarrow_{H\cap S_{n,T}}^{H}$ does not coincide with the trivial character of $H\cap S_{n,T}$. The conclusion follows from \Cref{le:vectors in kernel}. 
\end{proof}

\section*{Concluding remarks}
In 1991, Katona wrote about the wreath conjecture (``Sylvester's conjecture'' here refers to \Cref{thm:hypergraphdecomposition}):
``It seems to be hard to settle this conjecture.
Sylvester’s conjecture was earlier attacked by algebraic methods and an algebraic way of thinking. Baranyai’s brilliant
idea was to use matrices and flows in networks. This conjecture, however, seems to be too algebraic. One does not
expect to solve it without algebra. (Unless it is not true.)''

These words support the potential of the wreath matrix as a novel algebraic approach to the wreath conjecture. Moreover, the authors believe that the wreath matrix itself is an interesting object worth further investigation.

One particular property that the authors are curious about is the following.

\begin{question}\label{qu:distinct}
Are the eigenvalues $\lambda_1, \lambda_2, \ldots, \lambda_k$ from \Cref{thm:matrixresults} all distinct? 
\end{question}

\section*{Acknowledgemements}

The authors thank Béla Bollobás and Mark Wildon for helpful comments concerning the project and the manuscript. They also thank Miroslav Olšák for his help with generating decompositions of small families of sets $\groundsubs{n}{k}$ into disjoint wreaths, and Zsuzsanna Baran for additional comments on the manuscript.

The first author would like to acknowledge support by the EPSRC (Engineering and Physical Sciences Research Council), reference EP/V52024X/1, and by the Department of Pure Mathematics and Mathematical Statistics of the University of Cambridge. The second author would like to acknowledge support from Royal Holloway, University of London.

\bibliographystyle{abbrvnat}  
\bibliography{bibliography}

\begin{thebibliography}{20}
\providecommand{\natexlab}[1]{#1}
\providecommand{\url}[1]{\texttt{#1}}
\expandafter\ifx\csname urlstyle\endcsname\relax
  \providecommand{\doi}[1]{doi: #1}\else
  \providecommand{\doi}{doi: \begingroup \urlstyle{rm}\Url}\fi

\bibitem[Bahmanian(2014)]{bahmanian2014connected}
M.~A. Bahmanian.
\newblock Connected {B}aranyai’s theorem.
\newblock \emph{Combinatorica}, 34\penalty0 (2):\penalty0 129--138, 2014.

\bibitem[Bailey and Stevens(2010)]{BAILEY20103088}
R.~F. Bailey and B.~Stevens.
\newblock Hamiltonian decompositions of complete $k$-uniform hypergraphs.
\newblock \emph{Discrete Mathematics}, 310\penalty0 (22):\penalty0 3088--3095, 2010.

\bibitem[Baranyai(1974)]{baranyai1974factrization}
{\relax Zs}.~Baranyai.
\newblock On the factrization of the complete uniform hypergraphs.
\newblock \emph{Infinite and finite sets}, pages 91--108, 1974.

\bibitem[Ellis(2021)]{ellis2021intersection}
D.~Ellis.
\newblock Intersection problems in extremal combinatorics: Theorems, techniques and questions old and new, 2021.

\bibitem[Erd\H{o}s et~al.(1961)Erd\H{o}s, Ko, and Rado]{erdos1961intersection}
P.~Erd\H{o}s, C.~Ko, and R.~Rado.
\newblock Intersection theorems for systems of finite sets.
\newblock \emph{Quarterly Journal of Mathematics}, 12\penalty0 (1):\penalty0 313--313, 1961.

\bibitem[Godsil and Meagher(2015)]{GodsilMeagherEKR15}
C.~Godsil and K.~Meagher.
\newblock \emph{Erd\H{o}s--{K}o--{R}ado theorems: algebraic approaches}, volume 149.
\newblock Cambridge University Press, 2015.

\bibitem[He et~al.(2022)He, Huang, and Ma]{he2022non}
J.~He, H.~Huang, and J.~Ma.
\newblock A non-uniform extension of {B}aranyai's theorem.
\newblock \emph{arXiv:2207.00277}, 2022.

\bibitem[Huo et~al.(2015)Huo, Zhao, Feng, Yang, and Jirimutu]{huo2015jirimutu}
H.~Huo, L.~Zhao, W.~Feng, Y.~Yang, and Jirimutu.
\newblock Decomposing the complete $3$-uniform hypergraphs ${K}_n^{(3)}$ into {H}amiltonian cycles.
\newblock \emph{Acta Math. Sin.(Chin. Ser.)}, 58:\penalty0 965--976, 2015.

\bibitem[James(1978)]{JamesSymmetric78}
G.~D. James.
\newblock \emph{The representation theory of the symmetric groups}, volume 682 of \emph{Lecture Notes in Mathematics}.
\newblock Springer, Berlin, 1978.

\bibitem[Katona(1991)]{katona1991renyi}
G.~O.~H. Katona.
\newblock R{\'e}nyi and the combinatorial search problems.
\newblock \emph{Studia Sci. Math. Hungar}, 26\penalty0 (2-3):\penalty0 363--378, 1991.

\bibitem[Katona and Katona(2023)]{katona2023baranyai}
G.~O.~H. Katona and G.~Y. Katona.
\newblock Towards a {B}aranyai theorem with additional condition, 2023.

\bibitem[Kirkman(1847)]{kirkman1847problem}
T.~P. Kirkman.
\newblock On a problem in combinations.
\newblock \emph{Cambridge and Dublin Mathematical Journal}, 2:\penalty0 191--204, 1847.

\bibitem[Lov{\'a}sz(1979)]{lovasz1979shannon}
L.~Lov{\'a}sz.
\newblock On the {S}hannon capacity of a graph.
\newblock \emph{IEEE Transactions on Information theory}, 25\penalty0 (1):\penalty0 1--7, 1979.

\bibitem[Meszka and Rosa(2009)]{meszka2009decomposing}
M.~Meszka and A.~Rosa.
\newblock Decomposing complete $3$-uniform hypergraphs into {H}amiltonian cycles.
\newblock \emph{Australas. J Comb.}, 45:\penalty0 291--302, 2009.

\bibitem[Petr and Turek(2025)]{turek2024intervals}
J.~Petr and P.~Turek.
\newblock Intervals with a given number of falls across {D}yck paths of semilength $k$ and the wreath conjecture for $n=2k+1$.
\newblock \emph{arXiv:2501.07277}, 2025.

\bibitem[Saengchampa and Uiyyasathian(2024)]{SAENGCHAMPA2024114197}
C.~Saengchampa and C.~Uiyyasathian.
\newblock On {H}amiltonian decompositions of complete $3$-uniform hypergraphs.
\newblock \emph{Discrete Mathematics}, 347\penalty0 (12):\penalty0 114197, 2024.

\bibitem[Serre(1977)]{SerreRepresentations77}
J.-P. Serre.
\newblock \emph{Linear representations of finite groups}, volume~42 of \emph{Graduate Texts in Mathematics}.
\newblock Springer New York, NY, 1977.

\bibitem[Stanley(1999)]{StanleyEnumerativeII99}
R.~P. Stanley.
\newblock \emph{Enumerative combinatorics. {V}ol. 2}, volume~62 of \emph{Cambridge Studies in Advanced Mathematics}.
\newblock Cambridge University Press, Cambridge, 1999.

\bibitem[Xi(2024)]{xi2024variants}
Z.~Xi.
\newblock Variants of {B}aranyai's theorem with additional conditions.
\newblock \emph{arXiv:2410.08513}, 2024.

\bibitem[Xu and Wang(2002)]{xu2002hamiltonian}
B.~Xu and J.~Wang.
\newblock On the {H}amiltonian cycle decompositions of complete $3$-uniform hypergraphs.
\newblock \emph{Electronic Notes in Discrete Mathematics}, 11:\penalty0 722--733, 2002.

\end{thebibliography}

\appendix

\section{Coefficients $b_{l,j}$}\label{ap:coefficients}

For general value of the greatest common divisor $g$ of $n$ and $k$ when $k \nmid n$, in the spirit of \Cref{pr:bcoeficients} we compute that
\begin{align*}
b_{l,j} = \begin{cases}
    c_{l,0} \sum_{d=1}^{\frac{k}{g}} \left( \frac{n-k}{g} - d +1\right) \left( \frac{k}{g} - d +1\right) \left( \binom{dg}{l} -2\binom{dg-g}{l} + \binom{dg-2g}{l} \right) & \textnormal{if } j=0,\\
    c_{l,j} \sum_{d=1}^{\frac{k}{g}} \left( \frac{k}{g} - d +1\right) \sum_{\delta =1}^{d-1} \left( \binom{\delta g}{j} -\binom{\delta g-g}{j} \right)\left( \binom{(d-\delta)g}{l-j} -\binom{(d-\delta)g-g}{l-j} \right) & \textnormal{if } 0<j<l,\\
    c_{l,l} \sum_{d=1}^{\frac{k}{g}} \left( \frac{k}{g} - d +1\right)^2 \left( \binom{dg}{l} -2\binom{dg-g}{l} + \binom{dg-2g}{l} \right) & \textnormal{if } j=l.\\
    \end{cases}
\end{align*}
Here the coefficients $c_{l,j}$ are given by
\begin{align*}
c_{l,j} =
    \begin{cases}
        \frac{k! (n-k-l)! l!}{2(g!)^{\frac{n}{g}}} & \textnormal{if } j=0,\\
        \\
        \frac{j! (l-j)! (k-j)! (n-k-l+j)!}{(g!)^{\frac{n}{g}}} & \textnormal{if } 0<j<l,\\
        \\
        \frac{(n-k)!(k-l)! l!}{2(g!)^{\frac{n}{g}}} & \textnormal{if } j=l.\\
    \end{cases}
\end{align*}

  \par
  \medskip
  \begin{tabular}{@{}l@{}}%
    \textsc{DPMMS, University of Cambridge}\\
    \textsc{Wilberforce Road, Cambridge, CB3 0WA, United Kingdom}\\
    \textit{E-mail address}: \texttt{jp895@cam.ac.uk}
  \end{tabular}
  \par
  \medskip
  \begin{tabular}{@{}l@{}}
  \textsc{Department of Mathematics, Royal Holloway, University of London}\\
    \textsc{Egham Hill, Egham, Surrey TW20 0EX, United Kingdom}\\
    \textit{E-mail address}: \texttt{pkah149@live.rhul.ac.uk}
  \end{tabular}
\end{document}